\documentclass[]{article}
\usepackage{amssymb,amsmath,amsthm}

\newcommand{\Z}{{\mathbf Z}}
\newcommand{\R}{\mathbf{R}}

\renewcommand{\P}{\mathrm{P}}
\newcommand {\E}{\mathrm{E}}
\newcommand{\1}{{\bf 1}}
\renewcommand{\d}{\text{\rm d}}
\newcommand{\e}{\text{\rm e}}

\newtheorem{stat}{Statement}[section]
\newtheorem{proposition}[stat]{Proposition}

\newtheorem{theorem}[stat]{Theorem}
\newtheorem{lemma}[stat]{Lemma}
\theoremstyle{definition} 
\newtheorem{remark}[stat]{Remark}

\numberwithin{equation}{section}

\begin{document}

\title{\bf Zeros of a two-parameter random walk%
	\thanks{%
	Research supported in part by NSF grant DMS-0704024.}}
	
\author{Davar Khoshnevisan\\University of Utah
\and P\'al R\'ev\'esz\\Technische Universit\"{a}t Wien}

\date{July 2, 2009}
\maketitle
\begin{abstract}
	We prove that the number $\gamma_N$ of the zeros of a two-parameter
	simple random walk in its first $N\times N$ time steps is almost surely
	equal to $N^{1+o(1)}$ as $N\to\infty$. This is in contrast with our earlier
	joint effort with Z. Shi \cite{KRS:05}; that work shows that the number of 
	zero crossings in the first $N\times N$ time steps is $N^{(3/2)+o(1)}$
	as $N\to\infty$.
	We prove also that the number of zeros on the diagonal in the first
	$N$ time steps is $((2\pi)^{-1/2}+o(1)) \log N$ almost surely.\\

	
\end{abstract}

\section{Introduction}
Let $\{X_{i,j}\}_{i,j=1}^\infty$ denote i.i.d.\
random variables, taking the values $\pm 1$ with
respective probabilities $1/2$, and consider the
two-parameter random walk $\mathbf{S}:=\{S(n\,,m)\}_{n,m\ge 1}$ defined by
\begin{equation}
	S(n\,,m) := \sum_{i=1}^n\sum_{j=1}^m X_{i,j}
	\hskip.5in \text{for $n,m\ge 1$}.
\end{equation}

A lattice point $(i\,,j)$ is said to be a \emph{vertical crossing}
for the random walk $\mathbf{S}$ if $S(i\,,j)S(i\,,j+1)\le 0$. Let
$Z(N)$ denote the total number of vertical crossings in the
box $[1\,,N]^2\cap \mathbf{Z}^2$. 
A few years ago, together with Zhan Shi \cite{KRS:05} we proved
that with probability one,
\begin{equation}\label{eq:KRS:05}
	Z(N) = N^{(3/2)+o(1)}\qquad\text{as $N\to\infty$}.
\end{equation}
We used this result to describe an efficient method for plotting
the zero set of the two-parameter walk $S$; this was in turn
motivated by our desire to find good simulations of the level
sets of the Brownian sheet.

The goal of the present paper is to describe the rather different
asymptotic behavior of two other ``contour-plotting algorithms.''
Namely, we consider the total number of zeros in $[1\,,N]^2\cap\mathbf{Z}^2$:
\begin{equation}
	\gamma_N := \mathop{\sum\sum}\limits_{
	(i,j)\in[0,N]^2} \mathbf{1}_{\{S(i,j)=0\}},
\end{equation}
together with the total number of on-diagonal 
zeros in $[1\,,2N]^2\cap\mathbf{Z}^2$:
\begin{equation}
	\delta_N := \sum_{i=1}^N \mathbf{1}_{\{S(2i,2i)=0\}}.
\end{equation}
The main results are listed next.

\begin{theorem}\label{th:zeros}
	With probability one,
	\begin{equation}
		\gamma_N = N^{1+o(1)}\qquad
		\text{as $N\to\infty$}.
	\end{equation}
\end{theorem}

\begin{theorem}\label{th:diag}
	With probability one,
	\begin{equation}
		\lim_{N\to\infty} \frac{\delta_N}{\log N} = \frac{1}{(2\pi)^{1/2}},
	\end{equation}
	where ``$\log$'' denotes the natural logarithm.
\end{theorem}

The theorems are proved in reverse order of
difficulty, and in successive sections. 

\section{Proof of Theorem \ref{th:diag}}
Throughout, we need ordinary random-walk estimates. Therefore,
we use the following notation: 
Let $\{\xi_i\}_{i=1}^\infty$ be i.i.d.\ random variables,
taking the values $\pm 1$ with respective probabilities $1/2$,
and consider the one-parameter random walk $\mathbf{W}:=\{W_n\}_{n=1}^\infty$
defined by 
\begin{equation}\label{eq:W}
	W_n := \xi_1+\cdots+\xi_n.
\end{equation}
We begin by proving a simpler result.

\begin{lemma}\label{EN}
	As $N\to\infty$,\footnote{We always write $a_N=O(1)$
	to mean that $\sup_N |a_N|<\infty.$ Note the
	absolute values.}
	\begin{equation}
		\E \delta_N = \frac{1}{ (2\pi)^{1/2}}\log N +O(1).
	\end{equation}
\end{lemma}

Before we prove this, we recall some facts about simple random walks.

We are interested in the function,
\begin{equation}\label{eq:p}
	p(n) := \P\{ W_{2n}=0\}.
\end{equation}
First of all, we have the following, which is
a consequence of the inversion formula for Fourier transforms:
\begin{equation}\label{inversion}
	p(n) = \frac{1}{2\pi} \int_{-\pi}^\pi
	\left[ \cos (t) \right]^{2n}\, dt.
\end{equation}
Therefore, according to Wallis' formula \cite[eq.\ \textbf{6.1.49}, p.\ 258]{AS},
as $n\to\infty$,
\begin{equation}\label{p}
	p(n) = \frac{1}{(\pi n)^{1/2}} \left[
	1 - \frac{1}{8n} + \frac{1}{128 n^2} 
	-\cdots\right],
\end{equation}
in the sense of formal
power series.\footnote{Suppose $a_1,a_2,\ldots$
are non-negative series which $a_1(n)\le a_2(n)\le\cdots$.
Then please recall that
``$p(n)=a_1(n) - a_2(n) + a_3(n) - \cdots$'' is short-hand 
for ``$a_1(n)-a_2(n) \le p(n) \le 
a_1(n)-a_2(n)+a_3(n)$,'' etc.}

Next, we present a ``difference estimate.''

\begin{lemma}\label{diff}
	For all integers $n\ge 1$,
	\begin{equation}
		0\le p(n) - p(n+1) = O(n^{-3/2}).
	\end{equation}
\end{lemma}

\begin{proof}
	Because $0\le \cos^2 t\le 1$,
	\eqref{inversion} implies that
	$p(n)\ge p(n+1)$.  The remainder follows
	from \eqref{p} and a few lines of computations.
\end{proof}

\begin{proof}[Proof of Lemma \ref{EN}]
	Because $S(2i\,,2i)$ has the same
	distribution as $W_{4i^2}$, it follows that
	$\E\delta_N=\sum_{1\le i\le N} p(2i^2)$.
	The result follows readily from this and \eqref{p}.
\end{proof}

Next, we bound the variance of $\delta_N$.

\begin{proposition}\label{VarN}
	As $N\to\infty$,
	\begin{equation}
		\text{\rm Var}\, \delta_N = \frac{1}{(2\pi)^{1/2}} \log N + O(1).
	\end{equation}
\end{proposition}

\begin{proof}
	Evidently,
	\begin{equation}\label{Edelta^2}
		\E[\delta_N^2] = \E \delta_N + 2\mathop{
		\sum\sum}_{1\le i<j\le N} P(i\,,j),
	\end{equation}
	where
	\begin{equation}
		P(i\,,j) := \P \left\{ S(2i\,,2i)=0 ~,~ S(2j\,,2j)=0\right\},
	\end{equation}
	for $1\le i<j<\infty$.
	But $S(2j\,,2j)=S(2i\,,2i)+W_{i,j}$, where $W_{i,j}$ is a sum of
	$4(j^2-i^2)$-many i.i.d.\ Rademacher variables, and is
	independent of $S(2i\,,2i)$. Therefore, 
	\begin{equation}\label{Pij}
		P(i\,,j) = p(2i^2) p \left( 2(j^2-i^2)\right).
	\end{equation}
	According to Lemma \ref{diff}, $P(i\,,j)\ge
	p(2i^2)p(2j^2)$. Therefore, by \eqref{Edelta^2},
	\begin{equation}\begin{split}
		\E[\delta_N^2] &\ge \E \delta_N + 2\mathop{\sum\sum}_{
			1\le i<j\le N} p(2i^2)p(2j^2)\\
		&= \E \delta_N + \left( \E \delta_N \right)^2 -
			\sum_{1\le i\le N} p^2(2i^2).
	\end{split}\end{equation}
	Thanks to \eqref{p}, the final sum is $O(1)$. Therefore,
	Lemma \ref{EN} implies that
	\begin{equation}\label{Var:LB}
		\text{Var}\,\delta_N  \ge \frac{1}{(2\pi)^{1/2}}
		\log N + O(1).
	\end{equation}
	
	In order to bound the converse bound, we use Lemma \ref{diff} to find that
	\begin{equation}\begin{split}
		p\left( 2(j^2-i^2)\right) - p(2j^2)
			&= \sum_{2(j^2-i^2)\le \ell <
			2j^2} \left[  p(\ell)- p(\ell+1) \right]\\
		&\le c\sum_{2(j^2-i^2) \le \ell <
			2j^2} \frac{1}{\ell^{3/2}},
	\end{split}\end{equation}
	where $c$ is positive and finite, and does not depend on $(i\,,j)$.
	From this we can deduce that
	\begin{equation}\label{f:diff}\begin{split}
		p\left( 2(j^2-i^2)\right) - p(2j^2)
			&\le c'\frac{i^2}{(j^2-i^2)^{3/2}}\\
		&\le c'' \frac{i^2}{j^3(j-i)^{3/2}},
	\end{split}\end{equation}
	where $c'$ and $c''$ are positive and finite, and do not depend on $(i\,,j)$.
	Thus, we can find a positive and finite constant $c'''$ such that for
	all $n\ge 1$,
	\begin{equation}\label{ssf}
		\mathop{\sum\sum}_{1\le i<j\le N} p(2i^2)
		\left[ p\left( 2(j^2-i^2)\right) - p(2j^2)\right]
		\le c'''\mathop{\sum\sum}_{1\le i<j\le N}
		\frac{i}{j^3(j-i)^{3/2}}.
	\end{equation}
	We split the sum according to whether or not $j<2i$.
	First, we note that
	\begin{equation}
		\mathop{\sum\sum}_{\substack{
		1\le i,j<\infty\\i<j<2i}}
		\frac{i}{j^3(j-i)^{3/2}}
		\le \sum_{1\le i<\infty} \frac{1}{i^2}
		\sum_{i<j<\infty}\frac{1}{(j-i)^{3/2}}<\infty.
	\end{equation}
	Next, we note that
	\begin{equation}
		\mathop{\sum\sum}_{1\le i<2i<j<\infty}
		\frac{i}{j^3(j-i)^{3/2}} \le 2^{3/2}
		\mathop{\sum\sum}_{1\le i<2i<j<\infty}
		\frac{i}{j^{9/2}}<\infty.
	\end{equation}
	This and \eqref{ssf} together prove that
	\begin{equation}\begin{split}
		\mathop{\sum\sum}_{1\le i<j\le N}
			P(i\,,j) & \le
			\mathop{\sum\sum}_{1\le i<j\le N}
			p(2i^2) p(2j^2) + O(1)\\
		&= \left( \E \delta_N \right)^2 - \sum_{1\le i\le N}
			p^2(2i^2) + O(1)\\
		&= \left( \E \delta_N \right)^2 + O(1).
	\end{split}\end{equation}
	See \eqref{p}. This and \eqref{Edelta^2} together imply
	that the variance of $\delta_N$ is at most
	$O(1)$ + $\E \delta_N$. Apply Lemma \eqref{EN}
	and \eqref{Var:LB}, in conjunction, to finish the proof.
\end{proof}

\begin{proof}[Proof of Theorem \ref{th:diag}]
	Thanks to Proposition \ref{VarN}
	and the Chebyshev inequality, we can write the following:
	For all $\epsilon>0$,
	\begin{equation}
		\P\left\{ \left| \delta_N - \E\delta_N\right|\ge\epsilon \log N\right\}
		=O\left( \frac{1}{\log N}\right).
	\end{equation}
	Set $n_k:=[\exp(q^k)]$ for an arbitrary but fixed 
	$q>1$, and apply the Borel--Cantelli lemma to deduce that
	\begin{equation}
		\lim_{k\to\infty} \frac{\delta_{n_k}}{\log n_k}=
		\frac{1}{(2\pi)^{1/2}}.
	\end{equation}
	Let $m\to\infty$ and find $k=k(m)$ such that
	$n_k\le m<n_{k+1}$. Evidently,
	$\delta_{n_k}\le \delta_m\le \delta_{n_{k+1}}$. Also,
	$\log n_k\le \log n_{k+1} = (q+o(1))\log n_k$.
	Therefore, a.s.,
	\begin{equation}
		\limsup_{m\to\infty} \frac{\delta_m}{\log m}
		\le \limsup_{k\to\infty} \frac{\delta_{n_{k+1}}}{
		\log n_k} = \frac{q}{(2\pi)^{1/2}}.
	\end{equation}
	Similarly, a.s.,
	\begin{equation}
		\liminf_{m\to\infty} \frac{\delta_m}{\log m}
		\ge \liminf_{k\to\infty} \frac{\delta_{n_k}}{\log n_{k+1}}
		\ge \frac {1}{q(2\pi)^{1/2}}.
	\end{equation}
	Let $q\downarrow 1$ to finish.
\end{proof}

\section{Proof of Theorem \ref{th:zeros}}

We begin by proving the easier half of Theorem \ref{th:zeros};
namely, we first prove that with probability one,
$\gamma_N \le N^{1+o(1)}$.

\begin{proof}[Proof of Theorem \ref{th:zeros}: First Half]
	We apply \eqref{p} to deduce that as $N\to\infty$,
	\begin{equation}
		\E\gamma_N = \sum_{i=1}^N\sum_{j=1}^N \P\{S(i\,,j)=0\}
		=\sum_{i=1}^N\sum_{j=1}^N p(ij/2)
		\le \text{const}\cdot\left(\sum_{i=1}^N i^{-1/2}\right)^2,
	\end{equation}
	and this is $\le \text{const}\cdot N$.
	By Markov's inequality, 
	\begin{equation}
		\P\{\gamma_N\ge N^{1+\epsilon}\} \le \text{const}\cdot N^{-\epsilon},
	\end{equation}
	where the implied constant is independent of $\epsilon>0$ and $N\ge 1$.
	Replace $N$ by $2^k$ and apply the Borel--Cantelli lemma to deduce that
	with probability one, $\gamma_{2^k}< 2^{k(1+\epsilon)}$ for all $k$ sufficiently large.
	If $2^k\le N\le 2^{k+1}$ is sufficiently large [how large might be random], 
	then a.s.,
	\begin{equation}
		\gamma_N\le \gamma_{2^{k+1}} <2^{(k+1)(1+\epsilon)}
		\le 2^{k(1+2\epsilon)}\le N^{1+2\epsilon}.
	\end{equation}
	Since $\epsilon>0$ is arbitrary, this proves half of the theorem.
\end{proof}

The proof of the converse half is more delicate, and requires some
preliminary estimates.

For all $i\ge 1$ define
\begin{equation}\begin{split}
	\rho_1(i) &:= \min\left\{ j\ge 1:\ S(i\,,j)S(i\,,j+1)\le 0\right\},\\
	\rho_2(i) &:=\min\left\{ j\ge \rho_1(i):\ S(i\,,j)S(i\,,j+1)\le 0\right\},\\
	&\vdots\\
	\rho_\ell(i) &:=\min\left\{ j\ge\rho_{\ell-1}(i):\ S(i\,,j)S(i\,,j+1)\le 0\right\},\ldots\,.
\end{split}\end{equation}
These are the successive times of ``vertical upcrossings over time-level $i$.''
For all integers $i\ge 1$ and all real numbers $t\ge 1$,
let us consider
\begin{equation}
	f(i\,;t) := \max\left\{ k\ge 1:\ \rho_k(i) \le t\right\}.
\end{equation}
Then, it should be clear that
\begin{equation}\label{fZ}
	\sum_{i=1}^N f(i\,;N) = Z(N).
\end{equation}
where $Z(N)$ denotes the total number of vertical upcrossings in $[1\,,N]^2$;
see the introduction.

\begin{lemma}\label{lem:unif:hit}
	With probability one, if $N$ is large enough, then
	\begin{equation}
		\max_{1\le i\le N}
		f(i\,;N) \le N^{1/2+o(1)}.
	\end{equation}
\end{lemma}

\begin{remark}
	It is possible to improve the ``$\le$'' to an equality.
	In fact, one can prove that
	$f(1\,;N)=N^{1/2+o(1)}$ a.s., using the results of
	Borodin \cite{Borodin}; for further related results
	see \cite{CsorgoRevesz:85}. We will
	not prove this more general assertion, as we shall not
	need it in the sequel.
	\qed
\end{remark}

\begin{proof}
	Choose and fix two integers $N\ge 1$ and $i\in\{1\,,\ldots,N\}$.
	
	We plan to apply estimates from the proof of Proposition
	4.2 of \cite{KRS:05}, whose $\zeta_i(0\,,N)$
	is the present $f(i\,;N)$.
	
	After Koml\'os, Major, and Tusn\'ady \cite{KMT},
	we can---after a possible enlargement of the underlying
	probability space---find three finite and positive constants
	$c_1,c_2,c_3$ and construct a standard Brownian motion
	$\mathbf{w}:=\{w(t)\}_{t\ge 0}$ such that
	for all $z>0$,
	\begin{equation}\label{eq:KMT}
		\max_{1\le j\le N}
		\P\left\{ \left| S(i\,,j)-w(ij) \right|>c_1\log(ij)+z\right\}\le
		c_2\e^{-c_3 z}.
	\end{equation}
	The Brownian motion $\mathbf{w}$ depends on the fixed
	constant $i$, but we are interested only in its law, which is
	of course independent of $i$. In addition, the constants
	$c_1,c_2,c_3$ are universal.
	
	Fix $\epsilon\in(0\,,1/2)$ and $\delta\in(0\,,\epsilon/2)$,
	and consider the event
	\begin{equation}
		\mathcal{E}_N :=\left\{ \max_{1\le j\le N}\left| 
		S(i\,,j)-w(ij)\right| \le N^\delta\right\}.
	\end{equation}
	[We are suppressing the dependence of $\mathcal{E}_N$
	on $i$, as $i$ is fixed.]
	By \eqref{eq:KMT}, we can find a constant $c_4$---independent
	of $N$ and $i$---such that
	\begin{equation}\label{eq:P(E)}
		\P(\mathcal{E}_N) \ge 1 - c_4N^{-4}.
	\end{equation}
	Let $S(i\,,0):=0$ for all $i$. Then, almost surely on $\mathcal{E}_N$, we have
	\begin{equation}\label{eq:ub0}\begin{split}
		&\sum_{j=0}^{N-1} \1_{\{ S(i,j)\ge 0 ~,~ S(i,j+1)\le 0\}}\\
		&\hskip.6in\le \sum_{j=0}^{N-1} \1_{\{ w(ij)\ge -N^\delta~,~
			w(i(j+1))\le N^\delta\}}\\
		&\hskip.6in\le \sum_{j=0}^{N-1} \1_{\{w(ij)\ge 0~,~w(i(j+1))\le 0\}}
			+2\sup_{a\in\R}\sum_{j=0}^N\1_{\{a\le w(ij)\le a+N^\delta\}}.
	\end{split}\end{equation}
	This is equation (6.6) of \cite{KRS:05}.
	Now we use eq.\ (1.13) of Borodin \cite{Borodin} to couple $\mathbf{w}$ with another
	Brownian motion $\mathbf{B}:=\{B(t)\}_{t\ge 0}$ such that
	\begin{equation}\label{eq:Borodin}\begin{split}
		&\P\left\{ \left| \sum_{j=0}^{N-1}\1_{\{w(ij)\ge 0~,~
			w(i(j+1))\le 0\}} - \mu (N/i)^{1/2} L_1^0(B)\right|
			\ge c_5 N^{1/4}\log N\right\}\\
		&\hskip3.4in\le (c_5 N)^{-4},
	\end{split}\end{equation}
	where $\mu:=\E(\lfloor B^+(1)\rfloor)$, $c_5\in(0\,,1)$
	does not depend on $(i\,,N)$, and
	$L_1^0(\mathbf{B}) := \lim_{\eta\downarrow 0} (2\eta)^{-1}
	\int_0^1 \1_{\{|B(s)|\le \eta\}}\, \d s$
	denotes the local time of $\mathbf{B}$ at time $1$ at space value $0$. See also the derivation
	of \cite[eq.\ (6.10)]{KRS:05} for some detailed technical comments.
	
	It is well known that $\P\{L_1^0(\mathbf{B})\ge\lambda\}
	\le 2\e^{-\lambda^2/2}$ for all $\lambda>0$ \cite{Lacey}. In particular,
	$\P \{ L_1^0(\mathbf{B}) \ge N^\delta \}
	\le 2\exp(-N^\delta/2).$ Since $\delta<1/4$, this,
	\eqref{eq:P(E)}, and \eqref{eq:Borodin}
	together imply that
	\begin{equation}\label{eq:ub1}
		\P\left\{ \sum_{j=0}^{N-1}\1_{\{w(ij)\ge 0~,~ w(i(j+1))\le 0\}}
		\ge \frac{N^{(1/2)+\delta}}{i^{1/2}} \right\} \le c_6 N^{-4},
	\end{equation}
	where $c_6\in(1\,,\infty)$ is independent of $N$ and $i$.
	On the other hand, eq.\ (6.20) of \cite{KRS:05} tells us that
	we can find a constant $c_7\in(1\,,\infty)$---independent of
	$N$ and $i$---such that
	\begin{equation}
		\P\left\{ 2\sup_{a\in\R}\sum_{j=0}^N\1_{\{a\le w(ij)\le a+N^\delta\}}
		\ge \frac{N^{(1/2)+\delta}}{i^{1/2}} \right\}
		 \le c_7 N^{-4} + 2\exp\left(-N^{2\delta}\right).
	\end{equation}
	Since $i\ge 1$ and $\delta<1/4< 1/2$, this implies that
	\begin{equation}\label{eq:ub2}
		\P\left\{ 2\sup_{a\in\R}\sum_{j=0}^N\1_{\{a\le w(ij)\le a+N^\delta\}}
		\ge N^{(1/2)+\delta} \right\}
		\le c_7 N^{-4} + 2\exp\left(-N^{2\delta}\right).
	\end{equation}
	Now we combine \eqref{eq:ub0}, \eqref{eq:ub1}, and \eqref{eq:ub2}
	to deduce the following:
	\begin{equation}\begin{split}
		&\sum_{N=1}^\infty
			\P\left(\max_{1\le i\le N}\sum_{j=0}^{N-1}\1_{\{S(i,j)\ge 0~,~S(i,j+1)\le 0\}}
			\ge 2N^{(1/2)+\delta} ~;~ \mathcal{E}_N\right)\\
		&\qquad\le\sum_{N=1}^\infty\sum_{i=1}^N
			\P\left( \sum_{j=0}^{N-1}\1_{\{S(i,j)\ge 0~,~S(i,j+1)\le 0\}}
			\ge 2N^{(1/2)+\delta} ~;~ \mathcal{E}_N\right)\\
		&\qquad\le \sum_{N=1}^\infty \left( c_6N^{-3}
			+ c_7N^{-3}+2N\exp(-N^{2\delta})
			\right)\\
		&\qquad<\infty.
	\end{split}\end{equation}
	This and \eqref{eq:P(E)}, in turn, together imply that
	\begin{equation}
		\sum_{N=1}^\infty
		\P \left\{ \max_{1\le i\le N}
		\sum_{j=0}^{N-1}\1_{\{S(i,j)\ge 0~,~S(i,j+1)\le 0\}}
		\ge 2N^{(1/2)+\delta} \right\} <\infty.
	\end{equation}
	Since $-S$ is another simple walk on $\Z$, it follows that
	\begin{equation}
		\sum_{N=1}^\infty
		\P \left\{ \max_{1\le i\le N}
		f(i\,;N)
		\ge 2N^{(1/2)+\delta} \right\} <\infty.
	\end{equation}
	The lemma follows the Borel--Cantelli lemma,
	because $\epsilon$, and hence $\delta$,
	can be made arbitrarily small.
\end{proof}
\newpage

Consider the following random set of times:
\begin{equation}
	\mathcal{H}_N(\alpha\,,\beta) := 
	\left\{ 1\le i\le N^{1-\alpha} :\ f(i\,;N)>N^{(1/2)-\beta} \right\}.
\end{equation}

\begin{lemma}\label{lem:H}
	Choose and fix three positive constants $\alpha,\beta,\epsilon$
	such that $\beta>(\alpha/2)+\epsilon$. Then, the following happens
	a.s.: For all but a finite number of values of $N$,
	\begin{equation}
		\left| \mathcal{H}_N(\alpha\,,\beta)  \right|\ge N^{1 -(3\alpha/2)-2\epsilon},
	\end{equation}
	where $|\cdots|$ denotes cardinality.
\end{lemma}

\begin{proof}
	
	We apply \eqref{eq:KRS:05}, via \eqref{fZ}
	and Lemma \ref{lem:unif:hit},
	to see that with probability one,  the following holds
	for all but a finite number of values of $N$:
	\begin{equation}\begin{split}
		N^{(3(1-\alpha)/2)-\epsilon} &= \sum_{1\le i\le N^{1-\alpha}} f(i\,;N^{1-\alpha})\\
		&\le \sum_{1\le i\le N^{1-\alpha}}f(i\,;N)\\
		&=\sum_{i\in\mathcal{H}_N(\alpha,\beta)} f(i\,;N)
			+\sum_{\substack{1\le i\le N^{1-\alpha}:\\
			f(i,N)\le N^{(1/2)-\beta}}} f(i\,;N)\\
		&\le \left|\mathcal{H}_N(\alpha\,,\beta)\right|\cdot N^{(1/2)+\epsilon}
			+ N^{1-\alpha+{(1/2)-\beta}}.
	\end{split}\end{equation}
	The lemma follows because $\beta>(\alpha/2)+\epsilon$.
\end{proof}

Define
\begin{equation}
	U(i\,;\ell) := \mathbf{1}_{\{ S(i,\rho_\ell(i))S(i,1+\rho_\ell(i))=0\}}.
\end{equation}

The following is a key estimate in our proof of Theorem \ref{th:zeros}.

\begin{proposition}\label{pr:Bernstein}
	There exists a finite constant $c>0$ such that 
	for all integers $i,M\ge 1$,
	\begin{equation}
		\P\left\{ \sum_{\ell=1}^M U(i\,;\ell) \le \frac{cM}{i^{1/2}}\right\}
		\le \exp\left( -\frac{cM}{4i^{1/2}}\right).
	\end{equation}
\end{proposition}

Our proof of Proposition \ref{pr:Bernstein}
begins with an estimate for the simple walk.

\begin{lemma}\label{lem:LD}
	There exists a constant $K$ such that for
	all $n\ge 1$ and positive even integers $x\le 2n$,
	\begin{equation}
		\P\left( \left. W_{2n} = x\  \right|\, W_{2n} \ge x\right) \ge \frac{K}{n^{1/2}}.
	\end{equation}
\end{lemma}

\begin{proof}
	Let $\mathcal{P}_n(x)$ denote the conditional probability in the
	statement of the lemma. 
	Define the stopping times $\nu(x) := \min\{j\ge 1:\, W_{2j}=x\}$, and write
	\begin{equation}
		\mathcal{P}_n(x)  
		=\sum_{j=x/2}^n \frac{\P\left(\left. W_{2n} = x \ 
		\right|\, \nu(x)=2j\right)\cdot \P\{ \nu(x)=2j\}}{\P\{W_{2n}\ge x\}}.
	\end{equation}
	We first recall \eqref{eq:p}, and then apply
	the strong markov property to obtain
	$\P(W_{2n} = x \,|\, \nu(x)=2j ) = p(n-j)$.
	Thanks to \eqref{p}, we can find two constants $K_1$
	and $K_2$ such that $p(n-j)\ge K_1(n-j)^{-1/2}
	\ge K_1n^{-1/2}$ if $n-j\ge K_2$. On the other hand,
	if $n-j<K_2$, then $p(n-j)\ge K_3\ge K_3n^{-1/2}$.
	Consequently, 
	\begin{equation}\begin{split}
		\mathcal{P}_n(x) &\ge \frac{K_4}{n^{1/2} \P\{W_{2n}\ge x\}}
			\cdot \sum_{j=x/2}^n  \P\{\nu(x)=2j\}\\
		&= \frac{K_4}{n^{1/2}}
			\cdot \frac{\P\{\nu(x)\le 2n\}}{\P\{W_{2n}\ge x\}},
	\end{split}\end{equation}
	and this last quantity is at least
	$K_4 n^{-1/2}$ since $\{\nu(x)\le 2n\}\supseteq\{W_{2n}\ge x\}$.
\end{proof}

Here and throughout, let $\mathcal{F}(i\,;\ell)$ denote the $\sigma$-algebra
generated by the random variables $\{\rho_i(j)\}_{j=1}^\ell$
and $\{S(i\,,m)\}_{m=1}^{\rho_i(\ell)}$ [interpreted in the usual way,
since $\rho_i(\ell)$ is a stopping time for the infinite-dimensional
walk $i\mapsto S(i\,,\bullet)$]. Then we have the following.

\begin{lemma}\label{lem:LD1}
	For all $i,\ell\ge 1$,
	\begin{equation}\label{eq:LD1}
		\P\left(\left. S(i\,,1+\rho_\ell(i))=0\ \right|\,
		\mathcal{F}(i\,;\ell)\right) \ge \frac{K}{i^{1/2}},
	\end{equation}
	where $K$ was defined in Lemma \ref{lem:LD}.
\end{lemma}

\begin{proof}
	Let $\xi:=-S(i\,,\rho_\ell(i))$, for simplicity. According to the definition of
	the $\rho_\ell(i)$'s,
	\begin{equation}
		S(i\,,1+\rho_\ell(i)) \ge 0\qquad\text{almost surely on $\{\xi> 0\}$}.
	\end{equation}
	Consequently,
	\begin{equation}
		\Delta_{i,\ell}:=S(i\,,1+\rho_\ell(i)) - S(i\,,\rho_\ell(i))\ge \xi
		\qquad\text{almost surely on $\{\xi> 0\}$.}
	\end{equation}
	Clearly, the strong markov property of the infinite dimensional random walk
	$i\mapsto S(i\,;\bullet)$ implies that with probability one,
	\begin{equation}\begin{split}
		\P\left(\left. S(i\,,1+\rho_\ell(i))=0\ \right|\,
			\mathcal{F}(i\,;\ell)\right)
			&=\P\left(\left. \Delta_{i,\ell} = \xi \ \right|\,
			\mathcal{F}(i\,;\ell)\right)\\
		&\ge \P\left(\left. \Delta_{i,\ell}=\xi \ \right|\,
			\Delta_{i,\ell}\ge\xi \,;\,\xi\right)\mathbf{1}_{\{\xi>0\}}.
	\end{split}\end{equation}
	Therefore, we can apply Lemma \ref{lem:LD} together with  to deduce
	that \eqref{eq:LD1} holds a.s.\ on $\{\xi> 0\}$.
	Similar reasoning shows that the very same bound holds also a.s.\ on $\{\xi<0\}$.
\end{proof}

We are ready to derive Proposition \ref{pr:Bernstein}.

\begin{proof}[Proof of Proposition \ref{pr:Bernstein}]
	We recall the following form of Bernstein's inequality,
	as found, for example, in \cite[Lemma 3.9]{KRS:04}:
	{\it Suppose $J_1,\ldots,J_n$ are random variables, on
	a common probability space, that take values
	zero and one only. If there exists a nonrandom
	$\eta>0$ such that
	$\E (J_{k+1}\,|\, J_1\,,\ldots,J_k) \ge \eta$ for all $k=1,\ldots,n-1$. Then,
	that for all $\lambda\in(0\,,\eta)$,
	\begin{equation}\label{eq:Bernstein}
		\P\left\{ \sum_{i=1}^n J_i \le \lambda n\right\}
		\le \exp\left( - \frac{n(\eta-\lambda)^2}{2\eta}\right).
	\end{equation}
	}
	We apply the preceding with $J_\ell := U(i\,;\ell)$; Lemma
	\ref{lem:LD1} tells us that we can use 
	\eqref{eq:Bernstein} with $\eta:= Ki^{-1/2}$ and $\lambda:=\eta/2$ to 
	deduce the Proposition with $c:=K/2$.
\end{proof}

\begin{lemma}\label{lem:AN}
	Choose and fix two constants $a,b>0$ such that $1>a>2b$. 
	Then with probability one,
	\begin{equation}
		\min_{1\le i\le N^{1-a}} \sum_{1\le \ell\le N^\beta}
		U(i\,;\ell) \ge cN^{(a/2)-b},
	\end{equation}
	for all $N$ sufficiently large,
	where $c$ is the constant in Proposition \ref{pr:Bernstein}.
\end{lemma}

\begin{proof}
	Proposition \ref{pr:Bernstein} tells us that
	\begin{align}\nonumber
	&\P\left\{ \min_{1\le i\le N^{1-a}}\sum_{1\le\ell\le N^{(1/2)-b}} 
		U(i\,;\ell) 
		\le cN^{(a/2)-b} \right\}\\\nonumber
	& \hskip1in\le\P\left\{ \sum_{1\le\ell\le N^{(1/2)-b}} U(i\,;\ell) 
		\le \frac{cN^{(1/2)-b}}{i^{1/2}}\text{ for some
		$i\le N^{1-a}$}
		\right\}\\\nonumber
	&\hskip1in\le\sum_{1\le i\le N^{1-a}} \exp\left( -\frac{
		cN^{(1/2)-b}}{4 i^{1/2}}\right)\\
	&\hskip1in \le N^{1-a}\exp\left( -\frac{
		cN^{(a/2)-b}}{4}\right).
	\end{align}
	An application of the Borel--Cantelli lemma finishes the proof.
\end{proof}

We are ready to complete the proof of our first theorem.

\begin{proof}[Proof of Theorem \ref{th:zeros}: Second Half]
	Let us begin by choosing and fixing a small constant $\epsilon\in(0\,,1/2)$.
	Next, we choose and fix two more constants $a$ and $b$ such that
	\begin{equation}\label{ab}
		b\in(0\,,1/2)\quad\text{and}\quad
		a\in(2b\,,1).
	\end{equation}
	Finally, we choose and fix yet two more constants $\epsilon$ and $\alpha$
	such that
	\begin{equation}\label{eab}
		\alpha\in(a\,,1),\quad
		\beta\in\left(\frac\alpha 2+\epsilon~,~ b\right),\text{ and}\quad
		\frac{3\alpha}{2}-\frac a2 +b\le\epsilon.
	\end{equation}
	It is possible to verify that we can pick such $a$, $b$, $\alpha$,
	and $\beta$, regardless of how small $\epsilon$ is.
	
	Because $\alpha\in(a\,,1)$,
	\begin{equation}\begin{split}
		&\bigcap_{1\le i\le N^{1-a}} \left\{
			\sum_{1\le\ell\le N^{(1/2)-b}} U(i\,;\ell) >
			cN^{(a/2)-b} \right\}\\
		&\hskip.5in\subseteq \left\{
			\sum_{i\in\mathcal{H}_N(\alpha,\beta)}\sum_{1\le\ell\le N^{(1/2)-b}}
			U(i\,;\ell)\ge c N^{(a/2)-b}\left|
			\mathcal{H}_N(\alpha\,,\beta)\right|
			\right\}.
	\end{split}\end{equation}
	According to Lemma \ref{lem:H}, and since $\beta>(\alpha/2)+\epsilon$,
	$|\mathcal{H}_N(\alpha\,,\beta)|$ is at least
	$N^{1-(3\alpha/2)-2\epsilon}$, for all $N$ large.
	The preceding and Lemma \ref{lem:AN} together imply that
	with probability one,
	\begin{equation}\label{eq:U:LB}
		\sum_{i\in\mathcal{H}_N(\alpha,\beta)}\sum_{1\le\ell\le N^{(1/2)-b}}
		U(i\,;\ell) \ge c N^{1-(3\alpha/2)+(a/2)-b-2\epsilon},
	\end{equation}
	for all $N$ sufficiently large. Consequently, the following holds almost
	surely: For all but a finite number of values of $N$,
	\begin{equation}\begin{split}
		\gamma_N &= \sum_{i=1}^N\sum_{\ell=1}^{f(i,N)}
			U(i\,;\ell)\\
		&\ge\sum_{i\in\mathcal{H}_N(\alpha,\beta)}\sum_{1\le \ell\le N^{(1/2)-\beta}}
			U(i\,;\ell).
	\end{split}\end{equation}
	Since $\beta<b$, \eqref{eq:U:LB} implies that with probability one,
	the following holds for all but finitely-many values of $N$:
	\begin{equation}
		\gamma_N \ge c N^{1-(3\alpha/2)+(a/2)-b-2\epsilon},
	\end{equation}
	which is $\ge cN^{1-2\epsilon}$,
	thanks to the last condition of \eqref{eab}.
	Since $\epsilon$ is arbitrary, this completes our proof.
\end{proof}

\section{Questions on the distribution of zeros}
We conclude this paper by asking a few open
questions:

\begin{enumerate}
\item Let us call a point $(i\,,j)\in\Z^2_+$ \emph{even} if $ij$ is even.
	Define $Q_N$ to be the largest square in $[0\,,N]^2$
	such that $S(i\,,j)=0$ for every even point $(i\,,j)$ in $Q_N$.
	What is the asymptotic size of the cardinality of
	$Q_N\cap \Z^2$, as $N\to\infty$ along even integers?
	The following shows that this is a subtle question:
	One can similarly define $\tilde{Q}_N$ to
	be the largest square in $[0\,,N]^2$---with one vertex equal
	to $(N\,,N)$--such that $S(i\,,j)=0$ for all even
	$(i\,,j)\in \tilde{Q}_N$. [Of course, $N$ has to be even
	in this case.] In the present case, we estimate the size
	of $\tilde{Q}_N$ by first observing that if $N$ is even, then
	\begin{equation}\begin{split}
		&\P\left\{ S(N\,,N)=S(N+2\,,N+2)=0\right\}\\
		&= \P\left\{ S(N\,,N)=0\right\}\cdot
				\P\left\{ S(N+2\,,N+2)-S(N\,,N)=0\right\}\\
		&=(\text{const}+o(1)) N^{-3/2}\quad\text{as $N\to\infty$ along evens}.
	\end{split}\end{equation}
	Since the preceding defines a summable sequence, the Borel--Cantelli
	lemma tells us that $\#\tilde{Q}_N\le1 $ for all sufficiently-large
	even integers $N$.
\item Consider the number 
	$D_N := \sum_{i=1}^N \1_{\{S(i,N-i)=0\}}$
	of ``anti-diagonal'' zeros.
	It it the case that with probability one,
	\begin{equation}
		0<\limsup_{N\to\infty}\frac{\log D_N}{\log\log N}<\infty?
	\end{equation}
	At present, we can prove that
	$D_N\le (\log N)^{1+o(1)}$.

\item The preceding complements the following, which is not very hard to prove:
	\begin{equation}\label{eq:DN}
		\liminf_{N\to\infty} D_N=0\qquad\text{almost surely}.
	\end{equation}
	Here is the proof: According to the local central limit theorem,
	and after a line or two of computation,
	$\lim_{N\to\infty}\E (D_{2N})= (\pi/8)^{1/2}$.
	Therefore,  by Fatou's lemma,
	$\liminf_{N\to\infty}D_{2N}\le (\pi/8)^{1/2}<1$ with positive probability,
	whence almost surely by the Kolmogorov zero-one law [applied to the
	sequence-valued random walk $\{S(i\,,\bullet)\}_{i=1}^\infty$];
	\eqref{eq:DN} follows because $D_N$ is integer valued. We end by proposing 
	a final question related to
	\eqref{eq:DN}: Let $\{\mathcal{S}(s\,,t)\}_{s,t\ge 0}$ denote two-parameter
	Brownian sheet; that is, $\mathcal{S}$ is a centered gaussian process with continuous
	sample functions, and $\E[\mathcal{S}(s\,,t)\mathcal{S}(u\,,v)]=\min(s\,,u)\min(t\,,v)$
	for all $s,t,u,v\ge 0$.
	
	Define ``anti-diagonal local times,''
	\begin{equation}
		\mathcal{D}_t:= \lim_{\epsilon\to 0}\frac{1}{2\epsilon}
		\int_0^t \1_{\{|\mathcal{S}(s,t-s)|\le\epsilon\}}\, \d s
		\qquad\text{for $t>0$}.
	\end{equation}
	\begin{enumerate}
		\item Does $\{\mathcal{D}_t\}_{t>0}$ exist? Is it continuous?
		\item Is it true that $\mathcal{Z}:=\{t>0:\, \mathcal{D}_t=0\}$ is
			almost surely nonempty? That is, does the continuum-limit
			analogue of \eqref{eq:DN} hold?
			If $\mathcal{Z}$ is nonempty, then what is its Hausdorff dimension?
		\end{enumerate}
\item For all $\epsilon\in(0\,,1)$ and integers $N\ge 1$ define
	\begin{equation}
		E(\epsilon\,,N):=\left\{(i\,,j)\in[\epsilon N\,,N]^2:\, S(i\,,j)=0\right\}.
	\end{equation}
	It is not hard to verify that if $\epsilon\in(0\,,1)$ is fixed, then
	$E(\epsilon\,,N)=\varnothing$ for infinitely-many $N\ge 1$. This is because there
	exists $p\in(0\,,1)$---independent of $N$---such that for all $N$ sufficiently
	large,
	\begin{equation}
		\P\left\{ S(\epsilon N\,,N)\ge 2N\,,
		\max_{\epsilon N\le i,j\le N}
		\left| S(i\,,j)- S(\epsilon N\,,N)\right|\le N\right\}>p.
	\end{equation}
	Is there a good way to characterize which positive sequences $\{\epsilon_k\}_{k=1}^\infty$,
	with $\lim_{k\to\infty}\epsilon_k=0$,
	have the property that $E(\epsilon_N\,,N)\neq\varnothing$ eventually?
\item Let $\gamma_N'$ denote the number of points $(i\,,j)\in[0\,,N]^2$
	such that $S(i\,,j)=1$. What can be said about $\gamma_N-\gamma_N'$?
\item A point $(i\,,j)$ is a \emph{twin zero} if it is even and there exists
	$(a\,,b)\in\mathbf{Z}_+^2$ such that: (i) $0<|i-a|+|j-b|\le 100$ [say]; and (ii)
	$S(a\,,b)=0$. Let $d(\epsilon\,,N)$ denote the number of twin zeros that
	lie in the following domain:
	\begin{equation}
		D(\epsilon\,,N):=\left\{(i\,,j)\in\mathbf{Z}^2_+:\,
		\epsilon i < j < i/\epsilon\,, 1<i<N\right\}.
	\end{equation}
	Is it true that $\lim_{N\to\infty}d(\epsilon\,,N)=\infty$ a.s.\
	for all $\epsilon\in(0\,,1)$?
\end{enumerate}

\noindent\textbf{Acknowledgements.} This project have benefitted from 
several enlightening conversations with Professor Zhan Shi, whom we
thank gratefully.

\begin{small}

\vskip.4cm
\begin{footnotesize}
\noindent\textbf{Davar Khoshnevisan.}
\noindent Department of Mathematics, University of Utah,
	144 S 1500 E,
	Salt Lake City, UT 84112-0090, United States\\\texttt{davar@math.utah.edu}\\

\noindent\textbf{P\'al R\'{e}v\'{e}sz.} Institut f\"{u}r Statistik und 
	Wahrscheinlichkeitstheorie, Technische Universit\"{a}t  Wien, 
	Wiedner Hauptstrasse 8-10/107 Vienna, Austria\\\texttt{reveszp@renyi.hu}
\end{footnotesize}

\end{small}

\end{document}